\documentclass[12pt,twoside]{amsart}
\usepackage{amsmath}
\usepackage{amsthm}
\usepackage{amsfonts}
\usepackage{amssymb}
\usepackage{latexsym}
\usepackage{mathrsfs}
\usepackage{amsmath}
\usepackage{amsthm}
\usepackage{amsfonts}
\usepackage{amssymb}
\usepackage{latexsym}
\usepackage{geometry}
\usepackage{dsfont}
\usepackage[dvips]{graphicx}
\usepackage{color}
\usepackage[all]{xy}

\date{}
\allowdisplaybreaks[4] \footskip=15pt
\renewcommand{\uppercasenonmath}[1]{}

\usepackage{graphicx,amssymb}
\usepackage[all]{xy}
\usepackage{amsmath}

\allowdisplaybreaks
\usepackage{amsthm}
\usepackage{color}

\theoremstyle{plain}
\newtheorem{theorem}{Theorem}[section]
\newtheorem{proposition}[theorem]{Proposition}
\newtheorem{lemma}[theorem]{Lemma}
\newtheorem{corollary}[theorem]{Corollary}
\theoremstyle{definition}
\newtheorem{example}[theorem]{Example}
\newtheorem{definition}[theorem]{Definition}

\theoremstyle{definition}
\newtheorem*{acknowledgement}{Acknowledgement}

\theoremstyle{remark}
\newtheorem{remark}[theorem]{Remark}

\newcommand{\pf}{\noindent\begin {proof}}
\newcommand{\epf}{\end{proof}}

\newcommand{\Ker}{\mbox{\rm Ker}}
\newcommand{\Ext}{\mbox{\rm Ext}}
\newcommand{\Hom}{\mbox{\rm Hom}}

\newcommand{\Id}{\mathrm{Id}}

\def\ra{\rightarrow}

\def\Hom{{\rm Hom}}
\def\Ext{{\rm Ext}}

\def\ker{{\rm ker}}
\def\Ker{{\rm Ker}}

\def\SFT{{\rm SFT}}
\def\Im{{\rm Im}}

\def\Nil{{\rm Nil}}

\begin{document}
\begin{center}
{\large  \bf Nil$_{\ast}$-Noetherian rings}

\vspace{0.5cm}   Xiaolei Zhang$^{a}$%, Wei Qi$^{a}$ \\
%\bigskip

{\footnotesize a.\ School of Mathematics and Statistics, Shandong University of Technology, Zibo 255049, China\\

E-mail: zxlrghj@163.com\\}
\end{center}

\bigskip
\centerline { \bf  Abstract}
\bigskip
\leftskip10truemm \rightskip10truemm \noindent

In this paper, we say a ring $R$ is  Nil$_{\ast}$-Noetherian provided that any nil ideal is finitely generated. First, we show that the Hilbert basis theorem holds for Nil$_{\ast}$-Noetherian  rings, that is, $R$ is  Nil$_{\ast}$-Noetherian if and only if  $R[x]$ is  Nil$_{\ast}$-Noetherian, if and only if $R[[x]]$ is  Nil$_{\ast}$-Noetherian. Then we discuss some Nil$_{\ast}$-Noetherian properties on  idealizations and bi-amalgamated algebras. Finally, we give the Cartan-Eilenberg-Bass Theorem for  Nil$_{\ast}$-Noetherian rings in terms of Nil$_{\ast}$-injective modules and Nil$_{\ast}$-FP-injective modules. Besides, some examples are given to distinguish Nil$_{\ast}$-Noetherian rings, Nil$_{\ast}$-coherent rings and so on.
\vbox to 0.3cm{}\\
{\it Key Words:} Nil$_{\ast}$-Noetherian ring;  Hilbert basis theorem; Idealization; Bi-amalgamated algebra, Cartan-Eilenberg-Bass Theorem.\\
{\it 2020 Mathematics Subject Classification:}  16P40;16N40.

\leftskip0truemm \rightskip0truemm
\bigskip

Throughout this paper, all rings are  commutative  with identity and all modules are unitary. Let $R$ be a ring, we denote by $R[x]$ (resp., $R[[x]]$) the polynomial ring (resp., the formal series ring) in one variable over $R$. An element $r$ in $R$ is said to be nilpotent provided that $r^n=0$ for some positive integer $n$. The set of all nilpotent elements in $R$ is denoted by  $\Nil(R)$. An ideal $I$ of $R$ is said to be a nil ideal provided that any element in $I$ is nilpotent.

Recall that a ring $R$ is called to be a Noetherian ring if every ideal of $R$ is finitely generated. The concept of Noetherian rings, which was originally due to the mathematician Emmy Noether, is one of the most important topics that is widely used in ring theory, commutative algebra and algebraic geometry. The importance of Noetherian property was first shown in Hilbert basis theorem: if a ring $R$ is a Noetherian ring, then $R[x]$ and $R[[x]]$ are also Noetherian.  Noetherian rings also have many module-theoretic characterizations, such as the well-known  Cartan-Eilenberg-Bass Theorem states that a ring $R$  is Noetherian if and only if every direct sum of injective $R$-modules is injective, if and only if every direct limit of injective R-modules over a directed set is injective (see \cite[Theorem 4.3.4]{fk16} for example). Some new constructions of Noetherian rings, such as  idealizations and bi-amalgamated algebras, are also considered by many algebraists (see \cite{DW09,KLT17}). Several generalizations of Noetherian rings are introduced and studied by many algebraists. The famous generalization is the notion of coherent rings, i.e., rings in which any finitely generated ideal are finitely presented. For a further generalization, Xiang \cite{XO14} introduced the notions of $\Nil_{\ast}$-coherent rings in terms of nil ideals in 2014. A ring $R$ is said to be $\Nil_{\ast}$-coherent provided that any finitely generated nil ideal is finitely presented. Later in  2017, Alaoui  Ismaili et al.\cite{ADM17} studied the $\Nil_{\ast}$-coherent properties via idealization and amalgamated algebras under several assumptions.

The main motivation of this paper is to introduce and study the $\Nil_{\ast}$-Noetherian property of rings. Comparing with the concepts of Noetherian rings and  $\Nil_{\ast}$-coherent rings, we say a ring $R$ is  Nil$_{\ast}$-Noetherian provided that any nil ideal is finitely generated. It is important that the Hilbert basis theorem also holds for Nil$_{\ast}$-Noetherian  rings, that is, a ring $R$ is  Nil$_{\ast}$-Noetherian if and only if  $R[x]$ is  Nil$_{\ast}$-Noetherian, if and only if $R[[x]]$ is  Nil$_{\ast}$-Noetherian (see Theorem \ref{Hilbet} and Theorem \ref{Hilbet-series}). Utilizing these results, we study the idealization properties of Nil$_{\ast}$-Noetherian rings in Theorem \ref{idealization}, and then find  a $\Nil_{\ast}$-coherent ring that is not  Nil$_{\ast}$-Noetherian (see Example \ref{Nil-coh-not-Nil-Noe}). Surprisingly, Nil$_{\ast}$-Noetherian rings can also be not Nil$_{\ast}$-coherent (see Example \ref{nil-noe-not-nil-coh}).   By computing the nil-radical of the bi-amalgamated algebras under some assumptions (see Lemma \ref{nil-Ama}), we characterize  Nil$_{\ast}$-Noetherian properties of bi-amalgamated algebras in Proposition \ref{Amalgamation-1} and Theorem \ref{Amalgamation-2}. Finally, we give the Cartan-Eilenberg-Bass Theorem for  Nil$_{\ast}$-rings in terms of Nil$_{\ast}$-injective modules and Nil$_{\ast}$-FP-injective modules (See Theorem \ref{s-injective-ext}). In surprise, we  show that the direct limits of Nil$_{\ast}$-injective modules need not be Nil$_{\ast}$-injective  for Nil$_{\ast}$-Noetherian rings (see Remark \ref{dirlim-not}).

\section{Basic Properties of Nil$_{\ast}$-Noetherian Rings}

We begin with the concept of  Nil$_{\ast}$-Noetherian rings.
\begin{definition}\label{Nil-Noether}
 A ring $R$ is said to be a Nil$_{\ast}$-Noetherian ring provided that any nil ideal is finitely generated.
\end{definition}

Trivially, reduced rings and Noetherian rings are  Nil$_{\ast}$-Noetherian.

\begin{proposition}\label{nil-Noe-ideal}Let $R$ be a ring. Then the following assertions are equivalent:
\begin{enumerate}
\item  $R$ is  a Nil$_{\ast}$-Noetherian ring;
\item  $R$ satisfies the ascending chain condition on nil ideals;
\item  Every non-empty set of nil ideals of $R$ has a maximal element.
\end{enumerate}
\end{proposition}
\begin{proof} $(1)\Rightarrow(2)$  Let $I_1\subseteq I_2\subseteq \cdots\subseteq I_j\subseteq \cdots$ be an ascending chain condition on nil ideals. Set $I=\bigcup\limits_{j=1}^{\infty}I_j$. Then $I$ is also a  nil ideal. Hence $I$ is finitely generated. Consequently, there exists $k\in\mathbb{Z}$ such that $I=I_k$.

$(2)\Rightarrow(3)$ Let $\Gamma$ be a non-empty set of nil ideals of $R$. On contrary, suppose $\Gamma$ has no  maximal element. For any $I_1\in\Gamma$, there exits $I_2\in \Gamma$ such that  $I_1\subsetneq I_2$ as $I_1$ is not maximal. Continuing these steps, there is an strictly ascending chain condition on nil ideals $I_1\subsetneq I_2\subsetneq \cdots\subseteq I_n\subsetneq \cdots$, which contradicts(2).

$(3)\Rightarrow(1)$ Let $I$ be a nil ideal of $R$. Denoted by $\Gamma$ the set of all finitely generated sub-ideal of $I$. Then there is a maximal element $J$ in $\Gamma$. We claim $J=I$. Indeed, if there is an element $r\in I-J$. Then $J+Rr$ is a finitely generated sub-ideal of $I$ which is strictly lager than $J$, which is a contraction. So $I$ is finitely generated.
\end{proof}

\begin{lemma}\label{quo}
Let $R$ be a Nil$_{\ast}$-Noetherian ring. If  $I$ is a nil ideal of $R$, then $R/I$ is Nil$_{\ast}$-Noetherian.
\end{lemma}
\begin{proof} Let $K$ be  a nil ideal of $R/I$. Then $K=J/I$ for some $R$-ideal $J$ containing $I$. Then for any $j\in J$, there is an $n$ such that $j^n\in I$. Since $I$ is a nil ideal, $j^{nk}=0$ for some $k$. Hence $J$ is a nil ideal of $R$, and so is finitely generated. Hence $K$ is finitely generated $R/I$-ideal.
\end{proof}
Note that the condition ``$I$ is a nil ideal of $R$'' in Lemma \ref{quo}  cannot be removed.
\begin{example}\label{quo-not-nil-noe} Let $S=k[x_1,x_2,\cdots]$ be the polynomial ring over a field $k$ with countably infinite variables. Then $S$ is Nil$_{\ast}$-Noetherian.  Set the quotient ring $R=S/\langle x^2_i\mid i\geq 1\rangle$. Then $\Nil(R)=\langle \overline{x_1},\overline{x_2},\cdots \rangle$ is infinitely generated,  where $\overline{x_i}$ denotes the representative of $x_i$ in $R$ for each $i$. Hence $R$ is not  Nil$_{\ast}$-Noetherian.
\end{example}
\begin{proposition}\label{dire}
A finite direct product of rings $R=R_1\times \cdots\times R_n$ is Nil$_{\ast}$-Noetherian if and only if each $R_i$ is Nil$_{\ast}$-Noetherian $(i=1,\cdots,n).$
\end{proposition}
\begin{proof} It follows by $\Nil(R)=\Nil(R_1)\times \cdots\times \Nil(R_n)$ and $\Nil(R)$ is a Noetherian $R$-module if and only each $\Nil(R_i)$ is a Noetherian $R_i$-module ($i=1,\cdots,n$).
\end{proof}

\begin{remark} Suppose $R=\prod\limits_{i\in\Lambda}R_i$ is an infinite direct product of rings $R_i$. If $R$ is Nil$_{\ast}$-Noetherian, then trivially each direct summand $R_i$ is also Nil$_{\ast}$-Noetherian. However, the converse does not hold in general. Indeed, let $R_i=\mathbb{Z}_{p^i}$ the residue rings modulo $p^i$ where $p$ is a prime and $i$ an positive integer. Then each $R_i$ is Noetherian, and thus Nil$_{\ast}$-Noetherian. However, the direct product $R=\prod\limits_{i=1}^{\infty}R_i=\prod\limits_{i=1}^{\infty}\mathbb{Z}_{p^i}$ is not Nil$_{\ast}$-Noetherian since the nil ideal $\bigoplus\limits_{i=1}^{\infty}\Nil(\mathbb{Z}_{p^i})$ is not finitely generated.
\end{remark}

\begin{proposition}\label{ret}
Let $\phi:R\rightarrow S$ be a ring homomorphism making $R$ a module retract of $S$. If $S$ is Nil$_{\ast}$-Noetherian, then $R$ is also Nil$_{\ast}$-Noetherian.
\end{proposition}
\begin{proof} We can assume that $\phi$ is an inclusion map. Let $\psi:S\rightarrow R$ be an $R$-homomorphism such that $\psi\circ\phi=\Id_R$. Let $I$ be a nil ideal of $R$. Then $IS:=\{\sum\limits_{i=1}^tr'_is'_i\mid r'_i\in I, s'_i\in S\}$ is a nil ideal of $S$, and thus is finitely generated, say by $\{s_1,\cdots,s_n\}$. We will show $I$ is generated by $\{\psi(s_1),\cdots,\psi(s_n)\}$ as an $R$-module. Indeed, let $x$ be an element in $I$. Then $x=x1=\sum\limits_{i=1}^nr_is_i$ for some $r_i\in R$. Then $x=\psi\circ\phi(x)=\psi\circ\phi(\sum\limits_{i=1}^nr_is_i)=\sum\limits_{i=1}^nr_i\psi\circ\phi(s_i)=\sum\limits_{i=1}^nr_i\psi(s_i)$.
Hence $I$ is finitely generated. Consequently, $R$ is a Nil$_{\ast}$-Noetherian ring.
\end{proof}
Next, we will focus on the Nil$_{\ast}$-Noetherian properties of polynomial rings.
\begin{lemma}\label{nil-ele}\cite[Theorem 1.7.7(2)]{fk16}
Let $R$ be a ring. An element $f=\sum_{i=0}^na_ix^i\in $ is a nilpotent element in $R[x]$ if and only if each $a_i$ is a nilpotent element in $R$ $(i=0,\cdots,n).$
\end{lemma}

The well-known Hilbert Basis Theorem states that  a ring $R$ is a Noetherian ring if and only if $R[x]$ is a Noetherian  ring (see \cite[Theorem 4.3.15]{fk16})
\begin{theorem}\label{Hilbet} $($Hilbert Basis Theorem for Nil$_{\ast}$-Noetherian  rings-1$)$
Let $R$ be a ring. Then $R$ is a Nil$_{\ast}$-Noetherian  ring if and only if $R[x]$ is a Nil$_{\ast}$-Noetherian  ring.
\end{theorem}
\begin{proof}  Suppose $R[x]$ is a Nil$_{\ast}$-Noetherian ring. Let $I$ be a nil ideal of $R$. Then $I[x]$ is a nil ideal of $R[x]$, so is finitely generated, say is generated  by $\{h_1,\cdots,h_m\}$ as  $R[x]$-ideal. Let $a_i$ be the constant of $h_i$, then it is easy to verify that  $I$ is generated by $\{a_1,\cdots,a_m\}$.

Suppose  $R$ is a Nil$_{\ast}$-Noetherian ring. Let $J$ be a nil ideal of $R[x]$.  Set $I$ to  be the set of leading coefficients of polynomials $f$ in $J$. Then $I$ is a nil ideal by Lemma \ref{nil-ele}, and hence finitely generated. Write $I = Ra_1 +\cdots + Ra_k$ for each $a_i\in R$ and let $f_i\in J$ with
the leading coefficient $a_i$. Then $A = R[X]f_1 + \cdots + R[X]f_k \subseteq J$. Set $\deg( f_i) = n_i$ and $n = \max\{n_1,\cdots , n_k\}$. For any $f \in J$, write $f = aX^m +\cdots$. Then $a = r_1a_1 + \cdots + r_ka_k, r_i\in R$. If $m\geq n$, then $f':=f-\sum\limits_{i=1}^kr_iX^{m-n_i}f_i\in J$
with $\deg(f')< m$. If $\deg(f')\geq n$, we continue this process.  Hence there are polynomials $g\in A$ and $h\in R[X]$ with $\deg(h)< n$ such that $f=g+h$. Let $M[n]$ be an $R$-submodule of $R[X]$ generated by $1,X,\cdots,X^{n-1}$. Then $h=f-g\in M[n]\cap J$. Thus $J$ as an $R$-module is a sum of two $R$-submodules, that is, $J = A+J\cap M[n]$.

Claim that: {\bf the $R$-module $J \cap M[n]$ is finitely generated.}

Indeed, note that  $J \cap M[n]=\{f\in J\mid \deg(f)\leq n-1\}\cup \{0\}$. we will show the claim by induction on $n$. If $n=1$, then $J\cap M[1]$ is a nil ideal of $R$, so is finitely generated. If the claim holds for $n=k$, let $n=k+1$. Consider the following exact sequence:
\begin{center}
$0\ra J\bigcap M[k]\ra J\bigcap M[k+1]\ra L\ra 0.$
\end{center}
It is  easy to verify that  $L= J\bigcap M[k+1]/J\bigcap M[k]$  is isomorphic to a nil ideal of $R$ by Lemma \ref{nil-ele}, so is finitely generated. Since $J\bigcap M[k]$ is finitely generated by induction, we have $J\bigcap M[k+1]$ is also finitely generated, and the claim holds.

Hence $J\cap M = Rg_1 + \cdots + Rg_t$, where $g_j\in J\cap M$. Thus
\begin{align*}
 J\subseteq  & R[X]f_1 + \cdots + R[X]f_k+Rg_1 + \cdots + Rg_t \\
 \subseteq &R[X]f_1 + \cdots + R[X]f_k+R[X]g_1 + \cdots + R[X]g_t\subseteq J.
   \end{align*}
Consequently, $J=R[X]f_1 + \cdots + R[X]f_k+R[X]g_1 + \cdots + R[X]g_t$, which is finitely generated.
\end{proof}

Finally, we will focus on the Nil$_{\ast}$-Noetherian properties of formal series rings. Recall that an ideal $I$ of $R$ is said to be an $\SFT$ (strong finite type) ideal if there is a finitely generated sub-ideal $F$ of $I$ and an integer $n$ such that $a^n\in F$ for any $a\in I$. Trivially, every finitely generated ideal is an $\SFT$ ideal. For a ring $R$, we denote by $R[[x]]$ the formal series ring over $R$. By \cite[Theorem 11]{B81}, an element $f=\sum\limits^{\infty}_{i=0}a_ix^i\in R[[x]]$ is a nilpotent element, then each $a_i$ is nilpotent. However, the converse does not hold (see \cite[Example 2]{B81}).

\begin{lemma}\label{SFT-nil} \cite[Proposition 2.3]{HB11}
Let $R$ be a ring and an ideal $I\subseteq \Nil(R)$. Then the following assertions are equivalent:
\begin{enumerate}
\item  $I$ is  an  $\SFT$ ideal of $R$;
\item  $I[[x]]\subseteq \Nil(R[[x]])$;
\item  there exists an integer $k$ such that $r^k=0$ for any $r\in I$;
\item $I[[x]]$ is  an  $\SFT$ ideal of $R[[x]]$.
\end{enumerate}
\end{lemma}

\begin{lemma}\label{SFT-nil-1} \cite[Corollary 2.4]{HB11}
Let $R$ be a ring. Then the following assertions are equivalent:
\begin{enumerate}
\item  $\Nil(R)$ is  an  $\SFT$ ideal of $R$;
\item  $\Nil(R)[[x]]=\Nil(R[[x]])$;
\item  there exists an integer $k$ such that $r^k=0$ for any $r\in \Nil(R)$;
\item $\Nil(R[[x]])$ is  an  $\SFT$ ideal of $R[[x]]$.
\end{enumerate}
\end{lemma}

It is also well-known that a ring $R$ is a Noetherian ring if and only if $R[[x]]$ is a Noetherian  ring (see \cite[Theorem 4.3.15]{fk16}).
\begin{theorem}\label{Hilbet-series}  $($Hilbert Basis Theorem for Nil$_{\ast}$-Noetherian  rings-2$)$
Let $R$ be a ring. Then $R$ is a Nil$_{\ast}$-Noetherian ring if and only if $R[[x]]$ is a Nil$_{\ast}$-Noetherian  ring.
\end{theorem}
\begin{proof}  Suppose $R[[x]]$ is a Nil$_{\ast}$-Noetherian  ring. Then  $\Nil(R[[x]])$ is  an  finitely generated ideal, thus an  $\SFT$ ideal of $R[[x]]$. Then $\Nil(R)$ is  an  $\SFT$ ideal of $R$ by  Lemma \ref{SFT-nil-1}. Let $I$ be a nil ideal of $R$. Then  $I$ is also an  $\SFT$ ideal of $R$. So $I[[x]]$ is a nil ideal of $R[[x]]$  by  Lemma \ref{SFT-nil}, and thus is finitely generated, say is  generated by $\{h_1,\cdots,h_m\}$. Let $a_i$ be the constant of $h_i$, then it is easy to verify that  $I$ is generated by $\{a_1,\cdots,a_m\}$.

On the other hand, suppose  $R$ is a Nil$_{\ast}$-Noetherian ring. Let $J$ be a nil ideal of $R[[x]]$. We will prove $J$ is finitely generated. Write $J_r$ as an ideal of $R$ generated by the leading coefficients $a_r$ of $f=a_rx^r+a_{r+1}x^{r+1}+\cdots$ where $f\in J\cap x^rR[[x]]$. Note that each $a_r$ is nilpotent. So we have an increasing chain of nil ideals:$$J_0\subseteq J_1\subseteq J_2\subseteq\cdots\subseteq J_n\subseteq J_{n+1}\subseteq\cdots$$
Since $R$ is  Nil$_{\ast}$-Noetherian, there is an integer $s$ such that $J_n=J_s$ for any $n\geq s$ and $J_i$ is finitely generated for each $i$ with $0\leq i\leq s$ by Proposition \ref{nil-Noe-ideal}.
Now for each $i$ with $0\leq i\leq s$, take finitely many elements $a_{iv}\in R$ generating $J_i$, and take an element $g_{iv}\in J\cap x^iR[[x]]$ with $a_{iv}$ the leading coefficient of $g_{iv}$.

Claim that: {\bf these finitely many elements $g_{iv}$ generate $J$.}

Indeed, for each $f\in J$, we can take a linear combination $g_0$ of $g_{0v}$ with coefficients nilpotent in $R$ such that $f-g_0\in J\cap xR[[x]]$. Then take a linear combination $g_1$ of the $g_{1v}$ with coefficients nilpotent in $R$ such that $f-g_0-g_1\in J\cap x^2R[[x]]$.
Continuing these step, we get $f-g_0-g_1-\cdots-g_s\in J\cap x^{s+1}R[[x]]$. Since $J_{s+1}=J_s$, we can take a linear combination $g_{s+1}$ of $xg_{sv}$ with coefficients nilpotent in $R$ such that $f-g_0-g_1-\cdots-g_s-g_{s+1}\in J\cap x^{s+2}R[[x]]$.
Now, we proceed in the same way to get $g_{s+2},g_{s+3},\cdots$ For $i\leq s$, each $g_i$ is a linear combination of $g_{iv}$, and for each $i>s$, a combination of elements $x^{i-s}g_{sv}$.
For each $i\geq s$, we write $g_i=\sum\limits_{v}a_{iv}x^{i-s}g_{sv}$, and then for each $v$, we write $h_v=\sum\limits_{i=s}^{\infty}a_{iv}x^{i-s}$. So $$f=g_0+\cdots+g_{s-1}+\sum\limits_v h_vg_{sv}.$$
\end{proof}

\section{Nil$_{\ast}$-Noetherian properties on some ring constructions}
Some non-reduced rings are constructed by  the idealization $R(+)M$ where $M$ is an $R$-module (see \cite{H88}). Set $R(+)M=R\oplus M$ as an $R$-module, and then define
\begin{enumerate}
    \item ($r,m$)+($s,n$)=($r+s,m+n$),
    \item  ($r,m$)($s,n$)=($rs,sm+rn$).
\end{enumerate}
Under this construction, $R(+)M$ becomes a commutative ring with identity $(1,0)$. Note that $(0,m)^2=0$ for any $m\in M$. Now we characterize when  $R(+)M$ is a Nil$_{\ast}$-Noetherian  ring.

\begin{theorem}\label{idealization}
Let $R$ be a ring and $M$ an $R$-module. Then $R(+)M$ is a Nil$_{\ast}$-Noetherian  ring if and only if $R$ is  a Nil$_{\ast}$-Noetherian  ring and $M$ is a finitely generated $R$-module.
\end{theorem}
\begin{proof} For necessity, since $R\cong R(+)M/0(+)M$ and  $0(+)M$ is a nil ideal, $R$ is   Nil$_{\ast}$-Noetherian  by Lemma \ref{quo}. Since $0(+)M$ is a nil ideal, then $0(+)M$ is finitely generated $R(+)M$-module. It is easy to check that  $M$ is also a finitely generated $R$-module.

For sufficiency, suppose $M$ is generated by $n$ elements.  Set $K=\ker(R^n\twoheadrightarrow M)$. Then we have a short exact sequence $0\rightarrow  0(+)K\rightarrow  R(+)R^n\rightarrow R(+)M\rightarrow 0$ as $ R(+)R^n$-modules. By \cite[Proposition 2.2]{DW09}, $ R(+)R^n\cong R[x_1,\cdots,x_n]/\langle x_i^2\mid i=1,\cdots,n\rangle$ is a Nil$_{\ast}$-Noetherian by Theorem \ref{Hilbet} and Lemma \ref{quo}. Hence $R(+)M\cong R(+)R^n/0(+)K$ is also a Nil$_{\ast}$-Noetherian by Lemma \ref{quo} again.
\end{proof}

First, we give an example of Nil$_{\ast}$-Noetherian rings which are neither reduced nor Noetherian.

\begin{example}\label{Nil-Noether-not-NR}
Let $S=\prod\limits_{i}^{\infty}k$ be a countable copies of direct product of a field $k$, $e_i$ is an element in $S$ with the $i$-th component $1$ and others $0$. Set $R=S(+)Se_i$. Then $R$ is neither reduced nor Noetherian. However, since $S$ is a reduced ring and  $\Nil(R)=0(+)Se_i$ is a simple ideal, $R$ is  Nil$_{\ast}$-Noetherian by Theorem \ref{idealization}.
\end{example}

Recall from \cite{XO14} that a ring $R$ is called Nil$_{\ast}$-coherent provided that any finitely generated ideal in $\Nil(R)$ is finitely presented. Similar to the classical case, Nil$_{\ast}$-coherent rings are not Nil$_{\ast}$-Noetherian in general.
\begin{example}\label{Nil-coh-not-Nil-Noe}
Let $D$ be a non-field Noetherian GCD domain and $Q$ its quotient field. Set $R=D(+)Q/D$. Then by Theorem \ref{idealization}, $R$ is not Nil$_{\ast}$-Noetherian since $Q/D$ is not finitely generated. We will show $R$ is  Nil$_{\ast}$-coherent. Indeed, let $I$ be a finitely generated nil ideal of $R$.  Since $\Nil(R)=0(+)Q/D$, we can assume $I$ is generated by $\{(0,\frac{t_1}{s_1}+D),\cdots,(0,\frac{t_n}{s_n}+D)\}$ with all $gcd(s_i,t_i)=1$ and $s_i\not=1$. Consider the short exact sequence $0\rightarrow K_n\rightarrow R^n\rightarrow I\rightarrow 0$. We will show $K_n$ is finitely generated by induction on $n$. Indeed, for each $l\geq 1$, set $I_l=\langle (0,\frac{t_1}{s_1}+D),\cdots,(0,\frac{t_{l}}{s_{l}}+D)\rangle$. suppose $n=1$, then $K_1=(0:_R\frac{t_1}{s_1}+D)=\langle s_1\rangle(+)Q/D$ which is generated by $(s_1,0)$. Suppose it hold for $n= k\geq 1$. That is there is a natural exact sequence $0\rightarrow K_k\rightarrow R^{k} \rightarrow I_k\rightarrow0$. If $n=k+1$, set $a=(0,\frac{t_{k+1}}{s_{k+1}}+D)$.
There is an $R$-module $ K_{k+1}$ such that the following commutative diagram have  exact  rows and columns:
 $$\xymatrix@R=20pt@C=25pt{
0 \ar[r]^{} & K_k \ar@{^{(}->}[d]\ar[r]^{} &R^{k} \ar@{^{(}->}[d]\ar[r]^{} &I_k  \ar@{^{(}->}[d]\ar[r]^{} &  0\\
0 \ar[r]^{} & K_{k+1}\ar@{->>}[d]\ar[r]^{} & R^{k+1} \ar@{->>}[d]\ar[r]^{} &I_{k}+Ra\ar@{->>}[d]\ar[r]^{} &  0\\
0 \ar[r]^{} & (I_k:_RRa) \ar[r]^{} & R\ar[r]^{} & (I_{k}+Ra)/I_k\ar[r]^{} &  0,\\}$$
Since $(s_{k+1},0)\in (I_k:_RRa)-0(+)Q/D$, we have  $(I_k:_RRa)\supsetneq 0(+)Q/D$, and  so $(I_k:_RRa)=J'(+)Q/D$ for some nonzero finitely generated ideal  $J'$ of $D$ by \cite[Corollary 3.4]{DW09}. Hence $J'(+)Q/D$ is finitely generated by \cite[Proposition 2.6]{Z21-pvmr}. Consequently, $K_{k+1}$ is also finitely generated. It follows that $I$ is finitely presented, and thus  $R$ is  Nil$_{\ast}$-coherent.
\end{example}

Surprisingly, different with the classical case, Nil$_{\ast}$-Noetherian rings can also be non-Nil$_{\ast}$-coherent.

\begin{example}\label{nil-noe-not-nil-coh} Let $S=k[x_1,x_2,\cdots]$ be the polynomial ring over a field $k$ with countably infinite variables. Set $R=S/\langle x_1x_i\mid i\geq 1\rangle$. Then $\Nil(R)=\langle \overline{x_1}\rangle$ is the only non-trivial nil ideal of $R$, where $\overline{x_i}$ denotes the representative of $x_i$ in $R$ for each $i$. So $R$ is Nil$_{\ast}$-Noetherian. However,  since $(0:_R\overline{x_1})=\langle \overline{x_1},\overline{x_2},\cdots\rangle$ is infinitely generated,  $\Nil(R)$ is not finitely presented. Hence $R$ is not  Nil$_{\ast}$-coherent.
\end{example}

We recall the bi-amalgamated algebras constructed in \cite{KLT17}. Let $f:A\rightarrow B$ and $g:A\rightarrow C$ be two ring homomorphisms and let $J$ and
$J'$ be two ideals of $B$ and $C$, respectively, such that $f^{-1}(J)=g^{-1}(J')$. Denote by $I_0:=f^{-1}(J)=g^{-1}(J')$. The
bi-amalgamated algebra of $A$ with $(B, C)$ along $(J,J')$
with respect to $(f,g)$ is the subring of $B\times C$ given by:
$$A \bowtie^{f,g}(J,J'):= \{(f(a)+j, g(a)+j')\mid a\in A, (j, j')\in J\times J'\}.$$
The bi-amalgamation is determined by the following pull-back:
$$\xymatrix@R=20pt@C=15pt{
A \bowtie^{f,g}(J,J')\ar@{->>}[d]_{}\ar@{->>}[r]^{} &f(A)+J\ar[d]_{\alpha} \\
g(A)+J' \ar[r]^{\beta} &A/I_0,}$$
with $\alpha(f(a)+j)=\beta(g(a)+j')=\overline{a}$ for any $a\in A$. The following proposition is useful for continuation.
\begin{proposition}\label{pro-Ama}\cite[Proposition 4.1]{KLT17}
Let $f:A\rightarrow B$ and $g:A\rightarrow C$ be two ring homomorphisms and let $J$ and
$J'$ be two ideals of $B$ and $C$, respectively, such that $f^{-1}(J)=g^{-1}(J'):=I_0$. Let $I$ be an ideal of $A$. Then the following statements hold.
\begin{enumerate}
    \item $\frac{A \bowtie^{f,g}(J,J')}{I \bowtie^{f,g}(J,J')}\cong \frac{A}{I+I_0}$.
    \item $\frac{A \bowtie^{f,g}(J,J')}{0\times J'}\cong f(A)+J$ and $\frac{A \bowtie^{f,g}(J,J')}{J\times 0}\cong g(A)+J'$.
     \item $\frac{A}{I_0}\cong \frac{A \bowtie^{f,g}(J,J')}{J\times J'}\cong \frac{f(A)+J}{J}\cong\frac{g(A)+J'}{J'}.$
\end{enumerate}
\end{proposition}

\begin{lemma}\label{nil-Ama}
Let $f:A\rightarrow B$ and $g:A\rightarrow C$ be two ring homomorphisms and let $J$ and
$J'$ be two ideals of $B$ and $C$, respectively, such that $f^{-1}(J)=g^{-1}(J'):=I_0$. Suppose one of the following cases holds:
\begin{enumerate}
    \item $I_0$ is nil.
    \item $J\subseteq \Im(f)$ and $\Ker(f)$ is nil.
\end{enumerate}
Then  $$\Nil(A \bowtie^{f,g}(J,J'))=\Nil(A)\bowtie^{f,g}(J\cap \Nil(f(A)+J),J'\cap \Nil(g(A)+J')).$$
\end{lemma}
\begin{proof} (1) Let $\xi:=(f(a)+j, g(a)+j')\in \Nil(A \bowtie^{f,g}(J,J'))$. Then $f(a)+j\in \Nil(f(A)+J)$ and $g(a)+j'\in \Nil(g(A)+J')$. So $f(a^n)\in J$  for some positive integer  $n$. So $a^n\in I_0$, and thus is nilpotent. Hence $a\in \Nil(A)$.  So $f(a)\in \Nil(f(A))\subseteq \Nil(f(A)+J)$, and hence $j=(f(a)+j)-f(a)\in J\cap \Nil(f(A)+J)$. Similarly, $j'\in J'\cap \Nil(g(A)+J')$. Consequently, $\xi\in \Nil(A)\bowtie^{f,g}(J\cap \Nil(f(A)+J),J'\cap \Nil(g(A)+J'))$

On the other hand, let $a\in\Nil(A)$, $j\in J\cap \Nil(f(A)+J)$ and $j'\in J'\cap \Nil(g(A)+J')$. Then there is $k,m,n$ such that $a^k=j^m=j'^n=0$. Set  $\xi:=(f(a)+j, g(a)+j')\in A \bowtie^{f,g}(J,J')$. Then $\xi^{kmn}=0$, and hence $\xi\in \Nil(A \bowtie^{f,g}(J,J'))$.

(2) Let $\xi:=(f(a)+j, g(a)+j')\in \Nil(A \bowtie^{f,g}(J,J'))$. Then $f(a)+j\in \Nil(f(A)+J)$ and $g(a)+j'\in \Nil(g(A)+J')$. Since $J\subseteq \Im(f)$, there is an $x\in I_0$ such that $f(x)=j$. Note that $g(x)\in J'$. So $\xi:=(f(a+x), g(a+x)+j'-g(x))$. Since $\Ker(f)$ is nil, we have $a+x$ is nilpotent. Now, we claim that $j'-g(x)$ is nilpotent. Indeed, since $a+x$ is nilpotent, $g(a)+g(x)$ is nilpotent. So $j'-g(x)=(g(a)+j')-(g(a)+g(x))$ is nilpotent as $g(a)+j'$ is nilpotent. Hence $\xi\in \Nil(A)\bowtie^{f,g}(J\cap \Nil(f(A)+J),J'\cap \Nil(g(A)+J'))$. The other hand is the same as (1).
\end{proof}

It was proved in \cite[Proposition 4.2]{KLT17} that a ring $A \bowtie^{f,g}(J,J')$ is Noetherian if and only if $f(A)+J$ and $g(A)+J'$ are Noetherian. The main purpose of the rest section is to study Nil$_{\ast}$-Noetherian properties on bi-amalgamated algebras.

\begin{proposition}\label{Amalgamation-1}
Let $f:A\rightarrow B$ and $g:A\rightarrow C$ be two ring homomorphisms and let $J$ and
$J'$ be two ideals of $B$ and $C$, respectively, such that $f^{-1}(J)=g^{-1}(J'):=I_0$ is nil.  Then $A \bowtie^{f,g}(J,J')$ is Nil$_{\ast}$-Noetherian if and only if $f(A)+J$ and $g(A)+J'$ are Nil$_{\ast}$-Noetherian.
\end{proposition}
\begin{proof} Note that since $I_0$ is nil, $J$ and $J'$ are both nil. Then $J\times 0$ and  $0\times J'$ are also nil ideals of  $A \bowtie^{f,g}(J,J')$.  Suppose  $A \bowtie^{f,g}(J,J')$ is Nil$_{\ast}$-Noetherian.  Hence  $f(A)+J$ and $g(A)+J'$ are Nil$_{\ast}$-Noetherian by Lemma \ref{quo} and Proposition \ref{pro-Ama}(2). On the other hand, suppose $f(A)+J$ and $g(A)+J'$ are Nil$_{\ast}$-Noetherian. Let $I$ be a nil ideal of $A \bowtie^{f,g}(J,J')$ generated by $\{(f(a_i)+j_i,g(a_i)+j'_i)\mid i\in \Lambda\}$. Then the ideals $\langle f(a_i)+j_i\mid i\in \Lambda\rangle\subseteq \Nil(f(A)+J)$ and $\langle  g(a_i)+j'_i\mid i\in \Lambda\rangle\subseteq \Nil(g(A)+J')$ by Lemma \ref{nil-Ama}. So there is a finitely subset $\Lambda_0\subseteq \Lambda$ such that $\langle f(a_i)+j_i\mid i\in \Lambda\rangle =\langle f(a_i)+j_i\mid i\in \Lambda_0\rangle$ and $\langle  g(a_i)+j'_i\mid i\in \Lambda\rangle=\langle  g(a_i)+j'_i\mid i\in \Lambda_0\rangle$. Hence $I$ is a nil ideal of $A \bowtie^{f,g}(J,J')$ generated by $\{(f(a_i)+j_i,g(a_i)+j'_i)\mid i\in \Lambda_0\}$. So $A \bowtie^{f,g}(J,J')$ is Nil$_{\ast}$-Noetherian.
\end{proof}

\begin{theorem}\label{Amalgamation-2}
Let $f:A\rightarrow B$ and $g:A\rightarrow C$ be two ring homomorphisms and let $J$ and
$J'$ be two ideals of $B$ and $C$, respectively, such that $f^{-1}(J)=g^{-1}(J')$. If  $J\subseteq \Im(f)$, $J'\subseteq \Im(g)$, $\Ker(f)$ and $\Ker(g)$ are nil,  then $A \bowtie^{f,g}(J,J')$ is Nil$_{\ast}$-Noetherian if and only if $f(A)$ and $g(A)$ are Nil$_{\ast}$-Noetherian.
\end{theorem}
\begin{proof} Since $J\subseteq \Im(f)$, we have   $f(A)+J=f(A)$. Similarly, $g(A)+J'=g(A)$. Suppose $A \bowtie^{f,g}(J,J')$ is Nil$_{\ast}$-Noetherian.
Let $I$ be a nil ideal of $f(A)$ generated by $\{f(a_i)\mid a_i\in \Lambda\}$. Since $\Ker(f)$ is nil, each $a_i$ is also nilpotent. Consider the ideal $K$ of $A \bowtie^{f,g}(J,J')$ generated by $\{(f(a_i),g(a_i))\mid a_i\in \Lambda\}$. Then $K$ is nil and thus finitely generated. So there is a finite subset $\Lambda_0\subseteq \Lambda$ such that $K$  is generated by $\{(f(a_i),g(a_i))\mid a_i\in \Lambda_0\}$. Thus  $I$ is generated by  $\{f(a_i)\mid a_i\in \Lambda_0\}$. Hence $f(A)$ is Nil$_{\ast}$-Noetherian. Similarly, we have $g(A)$ is Nil$_{\ast}$-Noetherian.

Suppose  $f(A)$ and $g(A)$ are Nil$_{\ast}$-Noetherian. Let $I$ be a nil ideal of $A \bowtie^{f,g}(J,J')$ generated by $\{(f(a_i)+j_i,g(a_i)+j'_i)\mid i\in \Lambda\}\subseteq f(A)\times g(A)$. Then the ideals $\langle f(a_i)+j_i\mid i\in \Lambda\rangle\subseteq \Nil(f(A)+J)=\Nil(f(A))$ and $\langle  g(a_i)+j'_i\mid i\in \Lambda\rangle\subseteq \Nil(g(A)+J')=\Nil(g(A))$ with each $a_i\in \Nil(A)$ by Lemma \ref{nil-Ama}. So there is a finitely subset $\Lambda_0\subseteq \Lambda$ such that $\langle f(a_i)+j_i\mid i\in \Lambda\rangle =\langle f(a_i)+j_i\mid i\in \Lambda_0\rangle$ and $\langle  g(a_i)+j'_i\mid i\in \Lambda\rangle=\langle  g(a_i)+j'_i\mid i\in \Lambda_0\rangle$. Hence $I$ is a nil ideal of $A \bowtie^{f,g}(J,J')$ generated by $\{(f(a_i)+j_i,g(a_i)+j'_i)\mid i\in \Lambda_0\}$. So $A \bowtie^{f,g}(J,J')$ is Nil$_{\ast}$-Noetherian.
\end{proof}
Recall from \cite{DFF09} that, by setting $g=\Id_A: A\rightarrow A$ to be the identity  homomorphism of $A$, we denote by $A \bowtie^{f}J=A \bowtie^{f,\Id_A}(J,f^{-1}(J))$ and call it the amalgamated algebra of $A$ with $B$ along $J$
with respect to $f$.  Note that in this situation, the ring homomorphism $i:A\rightarrow A \bowtie^{f}J$, defined by $i(a)=(a,f(a))$ for any $a\in A$, is an $A$-module retract.

\begin{proposition}\label{nil-noe-ama}
Let $f:A\rightarrow B$ be a ring homomorphism and  $J$  an ideal of $B$.  If $A \bowtie^{f}J$ is Nil$_{\ast}$-Noetherian, so is $A$. Moreover, if $J\subseteq \Im(f)$ and  $\Ker(f)$ is nil, or $f^{-1}(J)$ is nil, then $A \bowtie^{f}J$ is Nil$_{\ast}$-Noetherian if and only if $A$ is Nil$_{\ast}$-Noetherian.
\end{proposition}
\begin{proof} Suppose $A \bowtie^{f}J$ is Nil$_{\ast}$-Noetherian.  Then $A$ is Nil$_{\ast}$-Noetherian by Proposition  \ref{ret}.  Since $f(A)\cong A/\Ker(f)$, the equivalence follows by Lemma \ref{quo}, Theorem \ref{Amalgamation-2} and Proposition \ref{Amalgamation-1} respectively.
\end{proof}

Recall from \cite{DF07} that, by setting  $f=\Id_A: A\rightarrow A$ to be the identity  homomorphism of $A$, we denote by $A \bowtie J=A \bowtie^{\Id_A}J$ and call it the amalgamated algebra of $A$ along $J$. By Proposition \ref{nil-noe-ama}, we obviously have the following result.
\begin{corollary} Let $J$ be an ideal of $A$. Then $A\bowtie J$ is Nil$_{\ast}$-Noetherian if and only if $A$ is  Nil$_{\ast}$-Noetherian.
\end{corollary}

\section{Module-theoretic characterizations of Nil$_{\ast}$-Noetherian rings}

We begin this section with the following two conceptions.
\begin{definition}
An $R$-module $M$ is said to be
\begin{enumerate}
    \item Nil$_{\ast}$-injective, provided that $\Ext_R^1(R/I,M)=0$ for any nil ideal $I$;
\item Nil$_{\ast}$-FP-injective, provided that $\Ext_R^1(R/I,M)=0$ for any finitely generated nil ideal $I$.
\end{enumerate}
\end{definition}

Trivially, injective modules are Nil$_{\ast}$-injective; Nil$_{\ast}$-injective modules and FP-injective  modules are Nil$_{\ast}$-FP-injective. The class of Nil$_{\ast}$-injective modules is closed under direct products; the class of Nil$_{\ast}$-FP-injective is is closed under direct sums and direct products.

The following result shows that non-trivial nil ideals  can never be projective.
\begin{lemma}\label{non-zeronil-not-proj}\cite[Proposition 6.7.12]{fk16}
Let $I$ be a non-zero nil ideal of a ring $R$. Then $I$ is not projective.
\end{lemma}

We first characterize rings over which all modules are Nil$_{\ast}$-injective or Nil$_{\ast}$-FP-injective.
\begin{proposition}
Let $R$ be a ring. Then the following statements are equivalent.
\begin{enumerate}
    \item $R$ is a reduced ring.
    \item Every $R$-module is Nil$_{\ast}$-injective.
     \item Every $R$-module is Nil$_{\ast}$-FP-injective.
\end{enumerate}
\end{proposition}
\begin{proof} $(1)\Rightarrow (2)\Rightarrow (3)$: Trivial.

$(3)\Rightarrow (1)$: Let $I$ be a finitely generated nil ideal of $R$. Consider the exact sequence $0\rightarrow I\rightarrow R\rightarrow R/I\rightarrow 0$. Then $I$ is Nil$_{\ast}$-FP-injective by $(3)$. So the exact sequence splits. Then $I$ is projective.  So each finitely generated nil ideal $I$, and thus $\Nil(R)$ is equal to $0$ by Lemma \ref{non-zeronil-not-proj}.
\end{proof}

We characterize $\Nil_{\ast}$-coherent rings in terms of Nil$_{\ast}$-FP-injective modules.

\begin{proposition}\label{s-FP-injective-ext}
Let $R$ be a ring. Then  $R$ is  Nil$_{\ast}$-coherent if and only if any direct limit of injective $R$-modules is Nil$_{\ast}$-FP-injective.
\end{proposition}

\begin{proof}
Let $I$ be a finitely generated nil ideal of $R$ and $\{M_i\}_{i\in I}$  a direct system of  injective $R$-modules. Then $\lim\limits_{\longrightarrow }\Ext^1_R(R/I,M_i)=0$ for each $i\in I$. Consider the short exact sequence $0\rightarrow I\rightarrow R\rightarrow R/I\rightarrow 0$, we have the following commutative diagram with rows exact:
$$\xymatrix{
 \ar[r]^{} &\lim\limits_{\longrightarrow } \Hom_R(R,M_i) \ar[d]_{\varphi_R}\ar[r]^{} & \lim\limits_{\longrightarrow }\Hom_R(I,M_i) \ar[r]^{}\ar[d]^{\varphi_I}& \lim\limits_{\longrightarrow }\Ext^1_R(R/I,M_i) \ar[r]^{}\ar[d]^{\varphi^1_{R/I}}&0 \\
 \ar[r]^{} &\Hom_R(R,\lim\limits_{\longrightarrow } M_i) \ar[r]^{} &\Hom_R(I,\lim\limits_{\longrightarrow } M_i)\ar[r]^{} & \Ext^1_R(R/I,\lim\limits_{\longrightarrow } M_i) \ar[r]^{} &0.
}$$
Since $R$ is a Nil$_{\ast}$-coherent ring, $I$ is a finitely presented ideal, then $\varphi_{I}$ is an isomorphism. Since $\varphi_{R}$ is an isomorphism, then $\varphi^1_{R/I}$ is also an isomorphism. Consequently, $\lim\limits_{\longrightarrow } M_i$ is a Nil$_{\ast}$-FP-injective $R$-module.

 Let $I$ be a finitely generated nil ideal, $\{M_i\}_{i\in I}$ a direct limit of  $R$-modules. Suppose $\alpha: I\rightarrow \lim\limits_{\longrightarrow }M_i$ is an $R$-homomorphism. For any $i\in I$, $E(M_i)$ is the injective  envelope of $M_i$. Then $\alpha$ can be extended to  be $\beta:R\rightarrow \lim\limits_{\longrightarrow }E(M_i)$. So there exists $j\in I$ such that $\beta$ can factor through $R\rightarrow E(M_j)$. Since the composition $I\rightarrow R\rightarrow E(M_j)\rightarrow E(M_j)/M_j$ becomes to be $0$ in the direct limit. We can assume  $I\rightarrow R\rightarrow E(M_j)$ factors through  $M_j$. Then $\alpha$ factor through $M_j$. So the natural homomorphism $\lim\limits_{\longrightarrow } \Hom_R(I,M_i)\rightarrow \Hom_R(I, \lim\limits_{\longrightarrow }M_i)$ is an epimorphism. So $I$ is a finitely presented ideal.
\end{proof}

\begin{theorem}\label{s-injective-ext} {\bf (Cartan-Eilenberg-Bass Theorem for  Nil$_{\ast}$-rings)}
Let $R$ be a ring. Then the following assertions are equivalent:
\begin{enumerate}
\item  $R$ is  Nil$_{\ast}$-Noetherian;
\item  any direct sum of Nil$_{\ast}$-injective $R$-modules is Nil$_{\ast}$-injective;
\item  any direct sum of injective $R$-modules is Nil$_{\ast}$-injective;
\item  any direct union of Nil$_{\ast}$-injective $R$-modules is Nil$_{\ast}$-injective;
\item  any direct union of injective $R$-modules is Nil$_{\ast}$-injective;
\item  any Nil$_{\ast}$-FP-injective  $R$-module is  Nil$_{\ast}$-FP-injective.
\end{enumerate}
\end{theorem}

\begin{proof}
$(1)\Rightarrow (6)$, $(4)\Rightarrow (2)\Rightarrow (3)$ and $(4)\Rightarrow (5)\Rightarrow (3):$  Trivial.

$(1)\Rightarrow (4):$ Let $\{M_i,f_{i,j}\}_{i<j\in \Lambda}$ be a direct system of   Nil$_{\ast}$-injective $R$-modules, where each $f_{i,j}$ is an inclusion map.
Set by $\lim\limits_{\longrightarrow}{M_i}$ the direct limit. Let $I$ be nil ideal of $R$, then $I$ is finitely generated. So $R/I$ is finitely presented. Consider the following commutative diagram with exact rows:
$$\xymatrix{
 \ar[r]^{} &\lim\limits_{\longrightarrow } \Hom_R(R,M_i) \ar[d]_{\varphi_R}\ar[r]^{} & \lim\limits_{\longrightarrow }\Hom_R(I,M_i) \ar[r]^{}\ar[d]^{\varphi_I}& \lim\limits_{\longrightarrow }\Ext^1_R(R/I,M_i) \ar[r]^{}\ar[d]^{\varphi^1_{R/I}}&0 \\
 \ar[r]^{} &\Hom_R(R,\lim\limits_{\longrightarrow } M_i) \ar[r]^{} &\Hom_R(I,\lim\limits_{\longrightarrow } M_i)\ar[r]^{} & \Ext^1_R(R/I,\lim\limits_{\longrightarrow } M_i) \ar[r]^{} &0.
}$$
By \cite[Theorem 24.10]{W91}, $\varphi_R$ and $\varphi_I$ are isomorphisms. So $\varphi^1_{R/I}$  is also an isomorphism. Hence, $\lim\limits_{\longrightarrow}{M_i}$ is Nil$_{\ast}$-injective.

$(3)\Rightarrow(1):$ On contrary, suppose $R$ is not Nil$_{\ast}$-Noetherian. Hence there is a non-finitely generated nil ideal $I$. So there is a nilpotent element $a_0\in I$ such that $a_0R\not=I$. Take $0\not=a_1\in I-a_0R$, then the nil ideal $a_1R+a_0R\not=I$. Take $a_2\in I-(a_1R+a_0R)$, then the nil ideal $a_1R+a_2R+a_0R\not=I$. Repeat these steps, we can get a strictly increasing chain of nil ideals:
$$a_0R+a_1R\subsetneq a_0R+a_1R+a_2R\subsetneq \cdots\subsetneq a_0R+a_1R+\cdots+a_nR\subsetneq \cdots$$
Set $A_i=\sum\limits_{j=0}^ia_jR$. According to \cite[Corollary 10.5]{AD92}, for any $A_i$, there exists maximal sub-ideal $C_i$ of $A_i$ satisfying $$C_1\subsetneq A_1\subseteq C_2\subsetneq A_2\subseteq\cdots\subseteq C_n\subsetneq A_n\subseteq\cdots$$ Set $E(A_i/C_i)$ the injective envelope of $A_i/C_i$.
Write $E=\bigoplus\limits_{i=1}^\infty E(A_i/C_i)$. Then $E$ is Nil$_{\ast}$-injective by (2). Set $A=\bigcup\limits_{i=1}^\infty A_i$. Then $A$ is a nil ideal. Set $\pi_i:A_i\twoheadrightarrow A_i/C_i$ the natural epimorphism and $\rho_i:A_i/C_i\hookrightarrow E(A_i/C_i)$ the inclusion map. Then there are homomorphism $f_i:A\rightarrow E(A_i/C_i)$ which is an extension of $\rho_i\circ \pi_i$.
Let $a\in A$ and $f:A\rightarrow E$ satisfies $f(a)=(f_i(a))_{i=1}^{\infty}$. Since $E$ is Nil$_{\ast}$-injective, then $f$ can be extended to be $g:R\rightarrow E$ with $g(r)=g(1)r$. Set $g(1)=(c_1,c_2,\cdots,c_n,0,\cdots)$. Take $a\in A$ such that  $a\in A_{n+1}-C_{n+1}$. Then $\rho_{n+1}\pi_{n+1}(a)\not=0$, and so $f_{n+1}(a)\not=0.$ However, $$ f(a)=g(a)=g(1)a=(c_1a,c_2a,\cdots,c_na,0,\cdots),$$ which is a contradiction. Hence, $R$ is  Nil$_{\ast}$-Noetherian.

$(6)\Rightarrow(2):$ Follows by any direct sum of  Nil$_{\ast}$-FP-injective  $R$-modules  is Nil$_{\ast}$-FP-injective.
\end{proof}

\begin{remark}\label{dirlim-not} We must remark that a ring $R$ is  Nil$_{\ast}$-Noetherian if and only if any direct union of Nil$_{\ast}$-injective $R$-modules is Nil$_{\ast}$-injective in Theorem \ref{s-injective-ext}. So  if any direct limit of Nil$_{\ast}$-injective $R$-modules is Nil$_{\ast}$-injective, then  $R$ is a Nil$_{\ast}$-Noetherian ring. However, the converse does not hold in general.  Indeed, let  $R$ is a Nil$_{\ast}$-Noetherian ring which is not Nil$_{\ast}$-coherent (see Example \ref{nil-noe-not-nil-coh}). Then  the class of  Nil$_{\ast}$-FP-injective  $R$-modules is equal to that of   Nil$_{\ast}$-FP-injective modules. However, by Proposition \ref{s-FP-injective-ext}, there exists a direct system of  Nil$_{\ast}$-injective $R$-modules, whose direct limit is not Nil$_{\ast}$-injective.
\end{remark}

\begin{acknowledgement}\quad\\
The  author was supported by the National Natural Science Foundation of China (No. 12061001).
\end{acknowledgement}

\end{document}